\documentclass{nhm}
\usepackage{amsmath}
  \usepackage{paralist}
  \usepackage{graphics} 
  \usepackage{epsfig} 
\usepackage[colorlinks=true]{hyperref}
\hypersetup{urlcolor=blue, citecolor=red}

\def\currentvolume{8}
 \def\currentissue{1}
  \def\currentyear{2013}
   \def\currentmonth{March}
    \def\ppages{1--xx}
     \def\DOI{10.3934/nhm.2013.8.xx}

\newtheorem{theorem}{Theorem}[section]

\newtheorem{lemma}[theorem]{Lemma}
\newtheorem{proposition}{Proposition}

\theoremstyle{definition}

\newtheorem{remark}{Remark}

\def\ve{\varepsilon}

\def\curl{\text{curl}}

\vspace{-4cm}

\title[Homogenized description of multiple GL vortices]
      {Homogenized description of multiple Ginzburg-Landau vortices pinned by small holes}

\author[Leonid Berlyand and Volodymyr Rybalko]{}


\email{berlyand@math.psu.edu}
\email{vrybalko@ilt.kharkov.ua}

\thanks{The work of  L. Berlyand  was supported  by NSF grant DMS--1106666, the work of V. Rybalko was supported in part by NSF grant DMS--1106666}

\begin{document}
\maketitle

\centerline{\scshape Leonid Berlyand}
\medskip
{\footnotesize
\centerline{Department of Mathematics}
\centerline{The Pennsylvania State University, University Park, PA 16802, USA}

} 

\medskip

\centerline{\scshape Volodymyr Rybalko}
\medskip
{\footnotesize
 \centerline{Mathematical Division}
   \centerline{B.Verkin Institute for Low Temperature Physics and Engineering}
  \centerline{of the National Academy of Sciences of Ukraine}
   \centerline{47 Lenin Ave., 61103 Kharkiv, Ukraine}
}

\bigskip


\begin{abstract}
We consider  a homogenization problem for the magnetic Ginzburg-Landau
functional in  domains with a large number of small holes. We
establish a scaling relation between sizes of holes and the
magnitude of the  external magnetic field when  the  multiple vortices
pinned by holes appear in nested subdomains and their  homogenized
density  is described by a hierarchy of variational problems. This
stands in sharp contrast with  homogeneous superconductors, where
all vortices are known to be simple.  The proof is based  on the
$\Gamma$-convergence approach applied to a  coupled
continuum/discrete variational problem: continuum in the induced
magnetic field and discrete in the unknown finite (quantized)
values of  multiplicity of vortices  pinned by holes.
\end{abstract}

\section{Introduction}

Vortices determine electromagnetic  properties of superconductors
that are important for practical applications (e.g., resistance).
A key practical issue is to decrease the energy dissipation in
superconductors, which occurs due to the motion of vortices.  This
dissipation can be suppressed by the pinning of vortices.  Moreover,  problems of pinning in superconducting thin films lead to  the analysis of a two-dimensional Ginzburg-Landau (GL)
energy functional in a domain with periodic or random arrays of holes  also called antidots in physical literature (see e.g.,  \cite{BerMilPee06}  and references therein).

In this work we consider  a  two-dimensional mathematical model of pinning of vortices
by many holes in  relatively small superconducting samples (comparable to the London depth).
The sample is subjected to a uniform magnetic field, which is weak so that the vortices do not appear in the bulk
of the superconducting sample and they may exist only in the holes.

Since  modern experimental  techniques allow for the creation of  very small
holes, the question arises  if any specific scaling relations
between the external magnetic field and the size of the holes can
lead to novel physical effects.  In particular, typical
experimental results lead to  uniform arrays of vortices in the
entire domain.  By contrast, in this work we derive a special
``critical" scaling relation when a  family of nested subdomains
with  pinned vortices of increasing multiplicity appears.

We next present a brief review  of relevant mathematical work. The
study of  pinning by a {\it finite  number of pinning sites} was
pioneered in \cite{LM}, where  a simplified GL model with no
magnetic field   and  discontinuous pinning term for a single
inclusion was considered.  The existence of $d$ vorticies of
degree $1$ inside the inclusion was established  for the Dirichlet
boundary  data with degree  $d$.  The results of \cite{LM}  were
subsequently generalized  for the magnetic GL functional
\cite{K1}, \cite{K2} and pinning by a single inclusion. A
comprehensive study of pinning by finitely many normal inclusions
and holes   for the magnetic GL functional was performed in
\cite{AB1,AB2}. More recently  pinning by finitely many holes
whose  sizes  goes to zero  as the Ginzburg-Landau parameter goes
to infinity was established in  \cite{DM} for  the SGL model.

Homogenization in the framework of magnetic GL model with
continuous oscillating pinning term was considered in the
pioneering work \cite{ASS}, where large number of vortices are
described by the homogenized vorticity density. Since  some
composite superconductors  are described by a discontinuous
pinning term, in subsequent works \cite{DMM,Dos11} homogenization
problems for such a term were address in the context of simplified
GL model and special  Dirichlet boundary conditions, which result
in either no vortices \cite{DMM} or $d$ vortices \cite {Dos11}.

In this work we study a  homogenization problem for a {\it large number of vortices}  and
large number of pinning holes that are described by a perforated domain $\Omega_\ve$
(which corresponds to  a discontinuous pinning  term). This problem  is  described by
the minimizers of the GL functional
\begin{equation}
\label{initialGL} GL(u,A)=\frac{1}{2}\int_{\Omega_\ve} (|\nabla u
-iAu|^2+\frac{\kappa^2}{2}(1-|u|^2)^2)\,
dx+\frac{1}{2}\int_\Omega({\rm curl} A-h_{ext}^\ve)\, dx.
\end{equation}
The unknowns here are the complex order parameter $u$ and the
vector potential of the magnetic field $A$, while $h_{ext}^\ve$ is
given constant external magnetic field (a positive scalar
quantity). 
The
domain $\Omega_\ve$ in (\ref{initialGL}) is obtained by
perforating a given simply connected bounded domain $\Omega\subset
\mathbb{R}^2$ by a large number $N_\ve$ of small holes.  Holes are
all identical disks with radius $\rho_\ve$ and have periodically
distributed centers $a_{j}^\ve$, with $\ve>0$ being a small
spatial period. We assume that
%
\begin{equation} \label{separation}
|\log \kappa| \gg |\log \rho_\ve |\ \text{and}\ \rho_\ve\ll \ve
\end{equation}
(hole radius much  greater than vortex core and much less than the spatial period $\varepsilon$). Moreover,
we consider the following scaling relations of the magnetic field and diameters of holes,
\begin{equation}
\label{scales}
h_{ext}^\ve=\sigma/\ve^2, \ \text{diam}(\omega_{j}^\ve)=2e^{-\gamma/\ve^2},
\end{equation}
where $\sigma$ and $\gamma$ are fixed positive numbers. This
scaling corresponds to the finite flux of the magnetic field over
each periodicity  cell $\int_{cell} h_{ext}^\ve dx=O(1)$. Note
that the effective core of a vortex pinned by a hole is  of the order
of the hole size  $\rho$, and therefore its energy is of order
$\log (1/\rho)$. Then if vortices inside holes have a finite
degree, we get $\rho=\exp{(-\gamma/\varepsilon^{2})}$, which  is a
much stronger separation than just $\rho \ll \ve$ in
\eqref{separation}. Our results show that the scaling
\eqref{scales} leads to a nonuniform spatial distribution of
multivortices, whereas the other scalings lead to a  simple
homogenization limit (no vortices for ``very small holes" and
constant uniform vorticity for larger holes). Our analysis shows
that the scaling \eqref{scales} is special because the  energy of
the vortices and the  bulk   energy of  the superconductor outside
the vortex cores are of the same order.

Recall that in homogeneous superconductors there are no vortices
if $0<h_{ext}^\ve<H_{c1}:=C_1 \log \kappa$ (below the first
critical field \cite{SS}). A further increase in the magnetic
field $C_1 \log \kappa <h_{ext}^\ve < C_2 \kappa^2$ results in a
lattice of simple (multiplicity one) vortices. Our results show
that in a superconductor with holes the interval
$0<h_{ext}^\ve<H_{c1}:=C_1 \log \kappa$  is  divided into two
subintervals: no vortices at all, and vortices inside holes but
not in the bulk of the superconductor.

Formal calculations show that under the scale separation condition
(\ref{separation}) there are no vortices in the bulk of the
domain $\Omega_\ve$ and the term with the integrand  $\frac{\kappa^2}{4}(1-|u|^2)^2$
can be effectively replaced by the constraint $|u|=1$.
Mathematical justification of this replacement based on the energy decomposition
will be presented elsewhere.   For  a  domain with finitely many holes of finite size,  analogous
replacement was rigorously justified in \cite{AB2}.  However,  even in the case of one  small hole the analysis  from  \cite{AB2} can not be carried out since  the smaller the hole the more energy is needed for pinning by this hole (that is the energy of the vortex in the hole  blows up as  the size of the hole goes to zero).

This leads us to study the homogenization limit for  the following  minimization problem
\begin{equation}
\label{min_prob}
M_\ve =\inf\{F_{\ve}(u,A);\ u\in
H^1(\Omega_\ve;S^1),\, A\in H^1(\Omega;\ \mathbb{R}^2)\}
\end{equation}
for the functional
\begin{equation}\label{harm_map_func}
F_{\ve}(u,A) = \frac{1}{2} \int_{\Omega_\ve}|\nabla u - iAu|^2 dx
+ \frac{1}{2} \int_{\Omega}(\curl A - h_{ext}^{\ve})^2 dx, \, \ve
> 0.
\end{equation}
Let $d_j^\ve$ be (integer) degrees of minimizer $u^\ve$ on
$\partial\omega_{j}^\ve$, then the main objective of this work  is
to obtain a {\it homogenized vorticity density} $D(x)$ defined as
the weak limit of measures $D(x)=w-\lim_{\ve\to 0}\ve^2 \sum
d_j^\ve\delta_{a_j^\ve}(x)$, that provides the limiting description of
$d_j^\ve$ as $\ve \to 0$.

For every fixed $\ve>0$ minimizers of $(u^\ve, A^\ve)$ of problem
(\ref{min_prob})-(\ref{harm_map_func})  can be expressed in terms of the degrees $d_j^\ve$
via the following problem for the  induced magnetic field $h^\ve=\curl A^\ve$ \cite{AB2}
\begin{equation}
\label{eq_h}
\begin{cases}
        -\Delta h^{\varepsilon} + h^{\varepsilon} = 0\ \mbox{in}\ \Omega_{\varepsilon}, \\
        h^{\varepsilon}(x) = h_{ext}^{\varepsilon}\ \mbox{on}\ \partial\Omega, \\
        h^{\varepsilon}(x) = H^{\varepsilon}_{ j}\ \mbox{ on }\ \omega_{\ve}^{j}, \ j=1,2,\dots,N_\ve,\\
        - \int_{\partial\omega_{j}^\ve} \frac{\partial h^{\varepsilon}}{\partial\nu} \; ds = 2\pi d_{j}^\ve -
        \int_{\omega_{j}^\ve} h^{\varepsilon}(x) dx,
\end{cases}
\end{equation}
where $H^{\varepsilon}_{ j}$ are unknown constants that are part of the problem. Note that if we know the
degrees $d_j^\ve$, then the minimum (\ref{min_prob}) is given by
\begin{equation}
\label{reduced_harm_map}
E_{\ve}(h^\ve)= \frac{1}{2} \int_{\Omega_\ve}|\nabla h^\ve|^2 dx
+ \frac{1}{2} \int_{\Omega}(h^\ve - h_{ext}^{\ve})^2 dx,
\end{equation}
and conversely, the minimum of (\ref{harm_map_func}) is obtained
by minimizing (\ref{reduced_harm_map}) in integer $d_j^\ve$, where
$h^{\varepsilon}(x;\{d_j^\ve\})$ is the unique solution of
(\ref{eq_h}). Thus, the  tuple $\{d_j^\ve\}$ is obtained by
minimizing \eqref{reduced_harm_map} over tuples of integers and
the  infinite-dimesional problem \eqref{min_prob} is reduced   the
following finite-dimensional (of dimension $N_\ve$) minimization
problem

\begin{equation}
\label{r} M_\ve= \inf\{E_\ve (h^\ve);\ h^\ve=h^\ve(x;\{d^\ve_j\})
\ \text{solves (\ref{eq_h}) for integer}\ d^\ve_j\}
\end{equation}


\begin{remark}\label{minimization} In the minimization problem \eqref{min_prob} $\{d_j^\ve\}$ are unknown integers defined as degrees of the $S^1$-valued function $u$.
  By contrast, problem \eqref{eq_h} can be solved for any unknown real-valued tuple $d_j^\ve,\; j=1,..., N_\ve$. However, in order to restore the minimizer $u_\ve$ from the minimizer $h^\ve$, it is necessary  for the constants  $d_j^\ve$ to be integers. Thus, the  condition of $d_j^\ve$ being integers (a quantization) is a {\it constraint} in the minimization problem for $h^\ve$.
\end{remark}

In broad terms our result can be described as follows. The
homogenized vorticity density $D(x)$ is a piecewise constant
function that takes increasing integer values on a family of
$j=j(\Omega,\gamma, \sigma)$ nested subdomains, and it is a  real
constant (not necessarily an integer) in the smallest (inner) domain.
This picture describes a rise of vortices of increasing
multiplicity. The precise formulation of these results is
presented below in Section \ref{4}.

Heuristic explanation of the formation of multivortices
in a superconductor with holes (unlike simple vortices in a homogenous SC)  can be described as follows.
If the number of vortices is much less than the number of holes, then all vortices are simple because of
repulsion and there are enough  available holes for all of them. If the number of vortices is comparable
with the number of holes (same order), then the collective effect takes over, namely,  vortices  are accumulated
in some  holes resulting  in the emergence of  multivortices.

\section{Homogenization (corrector) and compactness results}

\label{section_hom_and_comp}

We introduce res\-caled quantities
\begin{equation}
\label{rescaled_all}
\tilde h^\ve=\ve^2 h^\ve,\ \tilde h^\ve_{ext}=\ve^2 h^\ve_{ext},\
\tilde E_{\ve}(\tilde h_\ve)=\ve^4 E_{\ve}(\tilde h_\ve), \tilde M_\ve =\ve^4 M_\ve.
\end{equation}
%
Note that $h^\ve$ (and therefore $\tilde h_\ve$) is determined
uniquely by the tuple of integers $d_j^\ve$. Thus, abusing notation a
little, we may write $\tilde E_{\ve}(\tilde
h^\ve)=\tilde E_\ve (\{d_j^\ve\})$. Consider a minimizing tuple of
degrees $\{d_j^\ve\}$, so that the corresponding solution of
(\ref{eq_h}), rescaled according to (\ref{rescaled_all}),
satisfies $\tilde E_{\ve}(\tilde h_\ve)=\tilde M_\ve$.

First we obtain a priori bounds for these degrees $d_j^\ve$ in the following

\begin{lemma}\label{bound_deg}
Let $d_j^\ve$ be degrees of the minimizer of (\ref{min_prob}), then
\begin{equation}
\sum (d_j^\ve)^2\leq C/\ve^2,
\label{sum_d_j_squared}
\end{equation}
where $C$ is independent of $\ve$.
\end{lemma}

\begin{proof} The weak formulation of the problem for $\tilde h^\ve$ reads, find
$\tilde h^\ve\in H^1(\Omega)$ such that $\nabla \tilde h^\ve=0$  in all $\omega_j^\ve$ and
$\tilde h^\ve=\sigma$ on $\partial \Omega$, and
\begin{equation}
\label{eq_tilde_h}
\int_\Omega(\nabla \tilde h^\ve \cdot \nabla v+\tilde h^\ve v )dx-
2\pi\ve^2\sum  d_j^\ve v|_{\omega_j^\ve}
\\
=0
\end{equation}
holds for every test function $v\in H^1_0(\Omega)$ such that
$\nabla v=0$ in all $\omega_j^\ve$. In particular, if we choose
all $d_j^\ve=0$,  and set $v=\tilde h^\ve-\sigma$, we get  an a
priori bound (for $\tilde h$ that corresponds  to this choice of
degrees $d_j^\ve$)
\begin{equation*}
\|\tilde h^\ve\|_{H^1(\Omega)}:=\int_{\Omega_\ve}(|\nabla \tilde h^\ve|^2+ (\tilde h^\ve)^2) dx \leq C
\end{equation*}
therefore $\tilde M_\ve\leq C$, where $C$ is independent of $\ve$.
Hence for the minimizing tuple $\{d_j^\ve\}$ we have $\tilde
E_\ve(\tilde h^\ve)\leq C$ and thus  $\|\tilde
h^\ve\|_{H^1(\Omega)}\leq C$, with another independent of  $\ve$
constant $C$.
Now choose the test function
$v=\sum d_{j}^\ve L_j^\ve(x)/\log(2\rho_\ve/\ve)$ in (\ref{eq_tilde_h}), where
\begin{equation}
\label{def_L}
L_j^\ve(x)=
\begin{cases}
\log (2|x-a_j^\ve|/\ve)\ \text{if}\ x\in B_{\ve/2}(a_j^\ve)\setminus B_{\rho_\ve}(a_j^\ve)\\
\log(2\rho_\ve/\ve)\ \text{if}\ x\in B_{\rho_\ve}(a_j^\ve)\\
0\ \text{if}\ x\not\in B_{\ve/2}(a_j^\ve).
\end{cases}
\end{equation}

The heuristic idea behind the introduction of $L_j^\ve(x)$ can be
explained as follows. If the problem \eqref{eq_h} is rescaled by
$\ve^{-1}$, it becomes a Laplacian to the leading term, whose
solution  in an annular domain (with a single hole) and constant
values on the boundaries is $C_1+C_2 \log |x|$. Note also that
scaling in \eqref{def_L}  is such that the supports of
$L_j^\ve(x)$  are disjoint.

Now simple computations lead to the required bound,
$$
2\pi\ve^2\sum (d_j^\ve)^2=\int_\Omega(\nabla \tilde h^\ve \cdot \nabla v+\tilde h^\ve v )dx\leq
\sqrt{2\pi\ve^2(1+{o}(\ve))\sum (d_j^\ve)^2/\gamma}\, \|\tilde h^\ve\|_{H^1(\Omega)},
$$
i.e. $2\pi\ve^2\sum (d_j^\ve)^2\leq  (1+{o}(\ve))\|\tilde
h^\ve\|_{H^1(\Omega)}^2/\gamma\leq C$, where $\gamma$ is defined
in \eqref{scales}.
\end{proof}


It follows from Lemma \ref{bound_deg} that, up to extracting a subsequence,
\begin{equation}
\label{vorticity_conv}
\zeta^\ve=\ve^2 \sum d_j^\ve\delta_{a_j^\ve}(x)\rightharpoonup D(x)\ \text{as distributions, and}\
D(x)\in L^2(\Omega).
\end{equation}
>From now on let $\{d_j^\ve\}$ denote an arbitrary sequence of
tuples of integers such that (\ref{sum_d_j_squared}) and
(\ref{vorticity_conv}) hold, and $\tilde h^\ve$ is the function
associated to the tuple $\{d_j^\ve\}$, i.e. $\tilde h^\ve=\ve^2
h^\ve$, where $h^\ve$ is the solution of (\ref{eq_h}).

It is rather easy to see that under the above mentioned conditions we can pass to the limit
in (\ref{eq_tilde_h}) to get that $\tilde h^\ve $ converges weakly-$H^1$ to the solution
$\bar h$ of the homogenized
problem
\begin{equation}
\label{hom_eq_bar_h}
\begin{cases}
-\Delta \bar h +\bar h= 2\pi D(x)\ \text{in}\ \Omega\\
\bar h =\sigma\ \text{on}\ \partial \Omega.
\end{cases}
\end{equation}
(for details see Lemma \ref{hom_lemma} below). However this
weak-$H^1$ convergence is not sufficient for our principal goal of
describing the limiting vorticity $D(x)$. Actually, this will be done
by calculating the $\Gamma$-limit of the
functionals $\tilde E^\ve$, which requires convergence of energies
for the optimal lower bound (that matches the upper bound).  In
order to obtain the strong-$H^1$ convergence, we next introduce a
corrector.

Consider the ansatz,
\begin{equation}
\label{ansatz}
\tilde h^\ve(x)=\bar h^\ve(x)-\ve^2 \sum d_j^\ve L_j^\ve(x)=\bar h^\ve(x)+R^\ve,
\end{equation}
where functions $L_j^\ve(x)$ are given by (\ref{def_L}). The problem for $\bar h^\ve$
(in its weak form)
is, find
$\bar h^\ve\in H^1(\Omega)$ such that $\nabla \bar h^\ve=0$  in all $\omega_j^\ve$ and
$\bar h^\ve=\sigma$ on $\partial \Omega$, and
\begin{equation}
\label{eq_bar_h}
\int_\Omega(\nabla \bar h^\ve \cdot \nabla v\,+\bar h_\ve v )dx+\sum \int_{B_{\ve/2}(a_j^\ve)} v
R^\ve(x) dx\\
=2\ve\sum  d_j^\ve \int_{\partial B_{\ve/2}(a_j^\ve)}v\, d s
\end{equation}
holds for every test function $v\in H^1_0(\Omega)$ such that $\nabla v=0$ in all $\omega_j^\ve$.

\begin{lemma}
\label{correct_lem} Under conditions (\ref{sum_d_j_squared}),  (\ref{vorticity_conv})
functions $\bar h^\ve$ converge strongly-$H^1$ to $\bar h$, the
unique solution of  (\ref{hom_eq_bar_h}).
\end{lemma}

\begin{remark} Lemma \ref{correct_lem} shows that the function $R^\ve(x)=-\ve^2 \sum d_j^\ve L_j^\ve(x)$
is a corrector, so that $\bar h^\ve=\tilde h^\ve-R^\ve$ converges
strongly in $H^1(\Omega)$.
\end{remark}

\begin{proof}
Using  (\ref{sum_d_j_squared}) one shows that $R^\ve$ converges weakly-$H^1$ to zero, and therefore, up to extracting a
subsequence, $\bar h^\ve \rightharpoonup \bar h$. To prove that
$\bar h$ solves (\ref{hom_eq_bar_h}) we consider test functions
$v^\ve=v(x)+\sum \phi(x/\rho_\ve)(v(a_j^\ve)-v(x))$, where $v\in
C_0^\infty(\Omega)$ is an arbitrary function, and $\phi(x)$ is a
smooth cut-off function  such that $\phi=1$ if $|x|\leq 1$ and
$\phi=0$ if $|x|>2$. Set $v=v^\ve$ in  (\ref{eq_bar_h}) and pass
to the limit as $\ve\to 0$ to get
\begin{equation*}
\int_\Omega (\nabla \bar h\cdot\nabla v+\bar h v)\, dx=
2\pi\int_\Omega D(x)v \, dx.
\end{equation*}
Thus $\bar h$ solves (\ref{hom_eq_bar_h}).

Next we show that $\bar h^\ve$ converges strongly-$H^1$ to $\bar
h$. We set $v=\bar h^\ve-\sigma$ to obtain in the limit $\ve\to
0$,
\begin{equation*}
\limsup \int_\Omega\nabla \bar h^\ve \cdot \nabla \bar h^\ve
dx=2\limsup\ve\sum  d_j^\ve \int_{\partial B_{\ve/2}(a_j^\ve)}(\bar
h^\ve-\sigma) \, d s+\int_\Omega (\sigma- \bar h)\bar h dx.
\end{equation*}
By the Poincar\'e inequality, we have
$$
\int_{\Pi_j^\ve}\Bigl|\bar h^\ve- \frac{1}{\pi\ve}\int_{\partial B_{\ve/2}(a_j^\ve)}\bar h^\ve\, d s\Bigr|^2 dx\leq C\ve^2
\int_{\Pi_j^\ve}|\nabla\bar h^\ve|^2 dx
$$
and
$$
\int_{\Pi_j^\ve}\Bigl|\bar h^\ve- \frac{1}{\ve^2}\int_{\Pi_j^\ve}\bar h^\ve dx\Bigr|^2 dx\leq C\ve^2
\int_{\Pi_j^\ve}|\nabla\bar h^\ve|^2 dx,
$$
hence
$$
\frac{1}{\pi\ve}\int_{\partial B_{\ve/2}(a_j^\ve)}\bar h^\ve\, d s=
\frac{1}{\ve^2}\int_{\Pi_j^\ve}\bar h^\ve dx+ O(1)
\bigl(\int_{\Pi_j^\ve}|\nabla\bar h^\ve|^2 dx\bigr)^{1/2},
$$
where $\Pi_j^\ve$ is the cell with the center at $a_j^\ve$ and the
side length $\ve$. Therefore
$$
2\lim\ve\sum  d_j^\ve \int_{\partial B_{\ve/2}(a_j^\ve)}\bar h^\ve\, d
s=2\pi \int_\Omega D(x)\bar h\, dx,
$$
where we have used (\ref{sum_d_j_squared}), (\ref{vorticity_conv})
and the fact that $\bar h^\ve\to \bar h$ strongly in
$L^2(\Omega)$. Thus, taking into account (\ref{hom_eq_bar_h}), we
finally get
\begin{equation*}
\limsup \int_\Omega\nabla \bar h^\ve \cdot \nabla \bar h^\ve
dx=2\pi \int_\Omega D(x)(\bar h-\sigma) \, dx +\int_\Omega
(\sigma- \bar h)\bar h dx =\int_\Omega\nabla \bar h \cdot \nabla
\bar h dx.
\end{equation*}
This implies that $\bar h^\ve\to \bar h$ strongly in
$H^1(\Omega)$.
\end{proof}

As a corollary of Lemma \ref{correct_lem} we obtain the following
\begin{lemma} \label{hom_lemma}
Under conditions (\ref{sum_d_j_squared}), (\ref{vorticity_conv}) the following energy
expansion holds,
\begin{equation}
\label{corrector}
\tilde E_\ve(\tilde h^\ve)=\bar E_1(\bar h)+\pi\gamma \ve^2 \sum (d_j^\ve)^2+o(1),
\end{equation}
where
\begin{equation}
\label{first_hom_part}
\bar E_1(\bar h)=\frac{1}{2} \int_{\Omega}|\nabla \bar h|^2 dx + \frac{1}{2} \int_{\Omega}(\bar h - \sigma)^2 dx
\end{equation}
\end{lemma}

The energy expansion \eqref{corrector} will be used in the proof
of the $\Gamma-$ convergence result in Section
\ref{sectionLimVort}. Note that the second term in
\eqref{corrector} can also become vanishingly small, this is the
case when the external field is weak, $h^\ve_{ext}\ll \sigma_{\rm
cr1}/\ve^2$, where $\sigma_{\rm cr1}$ is the first critical value
given by \eqref{lambda_cr1}. On the other hand, if $\lim \ve^2
h^\ve_{ext}>\sigma_{\rm cr1}$, then the second term in expansion
\eqref{corrector} for minimizers $\tilde h^\ve$ of
\eqref{reduced_harm_map} is bounded below by a positive constant.

\begin{proof} Since $\bar h^\ve\to\bar h$ strongly in
$H^1(\Omega)$ while $R^\ve\to 0$ weakly in $H^1(\Omega)$, we have
$$
\tilde E_\ve(\tilde h^\ve)=
\bar E_1(\bar
h^\ve)+\frac{1}{2}\int_\Omega|\nabla R^\ve|^2\, dx+o(1)=\bar
E_1(\bar h)+\frac{1}{2}\int_\Omega|\nabla R^\ve|^2\, dx+o(1)
$$
A straightforward calculation of the second term in this
expansion yields (\ref{corrector}).
\end{proof}


\section{Limiting vorticity via $\Gamma$-convergence}
\label{sectionLimVort}

The main result of this work describing the limiting
vorticity is obtained by proving the $\Gamma$-convergence
of functionals $\tilde E_\ve$ with respect to the weak convergence (\ref{vorticity_conv}) of vorticity measures,
\begin{equation}
\label{Gamma_limit}
\tilde E_\ve (\{d_j^\ve\})\ \Gamma\text{-converge to}\ \bar E_0(D)= \bar E_1(\bar h)+\pi\gamma\int_\Omega\Phi (D(x))dx\
\text{as}\ \ve\to 0,
\end{equation}
where $\bar h$ is the unique
solution of (\ref{hom_eq_bar_h}).
More precisely we demonstrate that
\begin{itemize}
\item[(i)] ($\Gamma-liminf$ inequlity) if conditions (\ref{sum_d_j_squared}) and (\ref{vorticity_conv})
are satisfied, then
\begin{equation}
\label{liminf} \liminf \tilde E_\ve(\{d_j^\ve\})\geq  \bar E_0(D);
\end{equation}
\item[(ii)] ($\Gamma-limsup$ inequlity) $\forall D(x)\in L^2(\Omega)$ there is a (recovery) sequence of tuples
$\{d_j^\ve\}$ satisfying conditions (\ref{sum_d_j_squared}) and (\ref{vorticity_conv})
and such that
\begin{equation}
\label{limsup} \limsup \tilde E_\ve(\{d_j^\ve\})\leq \bar E_0(D).
\end{equation}
\end{itemize}

\begin{remark}
Note that condition (\ref{vorticity_conv}) implies  that the limit
$D(x)$ exists as a distribution. The condition
(\ref{sum_d_j_squared}) implies that $D(x)$ is, in fact, a
function from $L^2(\Omega)$.
 \end{remark}
The function $\Phi$ in the limit functional (\ref{Gamma_limit}) is
a continuous piecewise linear function such that $\Phi(d)=d^2$ at
integer points $d$. It describes the homogenized density of
energies of individual vortices, whereas $\bar E_1$ corresponds to
the interaction of a vortex with magnetic field due to other
vortices and external field.

Thanks to the energy expansion (\ref{corrector}), for
``$\liminf$" inequality
we need only to prove the lower bound
\begin{equation}
\label{aux_lower_bound}
\liminf \ve^2\sum (d_j^\ve)^2\geq \int_\Omega\Phi(D(x))\,dx.
\end{equation}
Since the left hand side of (\ref{aux_lower_bound}) is a nonlinear
(quadratic) function of $d_j^\ve$, we use an analog of Young
measures (see Appendix).

\subsection{Lower bound}

Spread the measure $\zeta^\ve$ (defined in \eqref{vorticity_conv}), which is the sum of point  masses, over periodicity
cells $\Pi_j^\ve$  by setting $D^\ve=d_j^\ve$ in
$\Pi_j^\ve$ ($\Pi_j^\ve$ is the  cell centered at $a_j^\ve$). Then represent $D^\ve$ as
\begin{equation}
\label{repres} D^\ve(x)=\sum_{k\in\mathbb{Z}} k \mu_k^\ve(x),\
\text{where}\ \mu_k^\ve (x)=\begin{cases}
1\; \text{in}\;  \Pi_{j}^{\ve}, \; \text{if} \;k=d_{j}^{\ve}\\
0,\; \text{otherwise}.
\end{cases}
\end{equation}
We extend $D^\ve$ and $\mu^\ve_k$, $k\not =0$, on $\Omega$ by
setting $D^\ve=\mu^\ve_k=0$ in  $\Omega\setminus \cup\Pi_j^\ve$ and also set $\mu_0=1$ in
$\Omega\setminus \cup\Pi_j^\ve$.
The functions $\mu_k^\ve(x)$ satisfy $\mu^\ve_k\geq 0$ and  $\sum \mu_k^\ve=1$ and therefore
form a partition of unity. Clearly,
we can extract a subsequence such that
\begin{equation}
\label{mu_k_converg}
\mu_k^\ve\rightharpoonup \mu_k\ \text{weakly in} \ L^2(\Omega)\ \forall k\in\mathbb{Z}.
\end{equation}
Thanks to the bound (\ref{sum_d_j_squared}) we have
\begin{equation}
\label{bound1}
\sum_{k\in\mathbb{Z}} k^2 \int_\Omega \mu_k^\ve dx=\ve^2\sum (d_j^\ve)^2\leq C.
\end{equation}
Hence the limit functions $\mu_k$ also form a partition of unity,
\begin{equation}
\label{tightness}
\mu_k\geq 0\ \text{and}\   \sum \mu_k=1.
\end{equation}
Indeed, the weak limit of non-negative functions is non-negative.
Analogously  $\sum \mu_k\leq 1$, while by (\ref{bound1}) we have
\begin{equation*}
\sum_{|k|\leq K} \int_\Omega \mu_k dx \geq |\Omega|-
\limsup_{\ve\to 0}\sum_{|k|>K}\int_\Omega \mu_k^\ve dx \geq |\Omega|-C/K^2,\  \forall K>0,
\end{equation*}
that yields $\sum \mu_k=1$.
Moreover, the function $D(x)$ defined in (\ref{vorticity_conv})
admits the representation
\begin{equation}
\label{D_represent}
D(x)=\sum k \mu_k(x)
\end{equation}
and  the equality in \eqref{bound1} implies that
\begin{equation}
\label{lower_bound_d}
\liminf \ve^2\sum (d_j^\ve)^2\geq \sum_{k\in\mathbb{Z}} k^2 \int_\Omega \mu_k(x).
\end{equation}
In order to obtain a lower bound in terms of $D(x)$, introduce
\begin{equation}
\label{toy_min} \Phi(D):=\min_{\{\mu_k\}_{k\in
\mathbb{Z}}}\left\{\sum_{k\in\mathbb{Z}} k^2 \mu_k; \ \mu_k\geq 0,
\ \sum \mu_k=1,\ \sum k \mu_k=D\right\}.
\end{equation}
for every real number $D$. Then \eqref{lower_bound_d} implies
the lower bound
 \begin{equation}
\label{lower_bound_extra}
\liminf \ve^2\sum (d_j^\ve)^2\geq  \int_\Omega \Phi(D(x)) dx
\end{equation}

The function $\Phi(D)$ can be computed as follows

 \begin{lemma}
 \label{Phi}
 $\Phi(D)=(2k+1)|D|-k-k^2$ if $k\leq|D|< k+1$, $k=0,1,2,\dots$.
Moreover, if $D=d$ is an integer, then the unique minimizing
sequence   is  $\mu_d=1$, $\mu_k=0$ \ $\forall k\not=d$. In the case $D$ is
non integer, represent $D$ as the convex hull of the two nearest
integers $d$ and $d+1$, $D=\alpha d+(1-\alpha) (d+1)$, then the
unique minimizing sequence is  $\mu_{d}=\alpha$,
$\mu_{d+1}=1-\alpha$, $\mu_k=0$ \ $\forall k\not\in\{d,d+1\}$.
\end{lemma}

\begin{proof} If $D=d$ is an integer then, clearly, $\Phi(D)\leq d^2$, and by
the Cauchy-Schwarz inequality we have
\begin{equation}\label{bound2}
d^2=(\sum k\mu_k )^2=(\sum k\sqrt{\mu_k}\sqrt{\mu_k})^2\leq \sum
k^2 \mu_k =\Phi(D).
\end{equation}
Thus $\Phi(D)=d^2$ and the minimizing  sequence is $m_d=1$, $m_k=0$
$\forall k\not=d$.

Consider now the case when $D$ is non integer and $D = \alpha d +
(1 - \alpha) (d+1)$. Let $\{\mu_k\}$ be a minimizing sequence. If
$\mu_k > 0$ for some $k < d$ then there is $\mu_l> 0$ for some
$l\geq d+1$. Decrease $\mu_k$ and $\mu_l$ by a sufficiently small
$\delta>0$ and increase $\mu_{k+1}$ and $\mu_{l-1}$ by $\delta$.
This modification changes neither \eqref{tightness} nor
\eqref{D_represent} but decreases the value of the functional $\sum k^2
\mu_k$. Therefore $\mu_k=0$  $\forall k\not\in\{d,\,d+1\}$. The
case when $\mu_k > 0$ for some $k > d+1$ is similar. It follows
from \eqref{tightness} and \eqref{D_represent} that
$\mu_{d}=\alpha$, $\mu_{d+1}=1-\alpha$ and straightforward
calculations yield the result.
\end{proof}

\subsection{Upper bound}

In order to complete the proof of $\Gamma$-convergence (\ref{Gamma_limit})
we have to show the $\limsup$-inequality, i.e.,  given $D\in L^2(\Omega)$, we need
to construct a (recovery) sequence of tuples $\{d_j^\ve\}$ satisfying
the boundedness condition (\ref{sum_d_j_squared}), that converge to $D(x)$
in the sense of (\ref{vorticity_conv}) and
satisfy inequality (\ref{limsup}).


The limiting functional $\bar E_0(D)$ is continuous with respect
to the strong convergence in $L^2(\Omega)$ therefore it is
sufficient to establish the $\limsup$-inequality for $D\in
C^\infty_0(\Omega)$ and then use density of $C^\infty_0(\Omega)$
in $L^2(\Omega)$. Moreover, due to Lemma \ref{hom_lemma} in order to
establish \eqref{limsup}, it is sufficient to  construct a
recovery sequence  of tuples $\{d_j^\ve\}$ that satisfies
\begin{equation}
\label{aux_upper_bound} \ve^2\sum
(d_j^\ve)^2\leq\int_\Omega\Phi(D(x))\,dx.
\end{equation}
Note that we not only need the converge of tuples in the sense of
(\ref{vorticity_conv}) but more importantly we need the
convergence of energies which does not follow from
(\ref{vorticity_conv}).

The key issue in the construction of the upper bound is that
different configurations of vortices  may lead to the same
homogenized vorticity $D(x)=\sum k\mu_k(x)$, however, these
configurations can be distinguished by $\sum k^2\mu_k(x)$ (which
is equal to $\Phi(D)$ for  the optimal $\mu_k$ given by Lemma
\ref{Phi}).\begin{footnote} {This important point can be
illustrated heuristically  as follows. In domain $\Omega$ consider
$\ve$-periodic checkerboard microstructure with black cells having
degree $-1$ and white cells having degree $1$. Then in the limit
$\ve\to 0$ we get  $\mu_{-1}=\mu_{1}=1/2$.  Therefore the first
moment is zero, $D = (-1)\frac{1}{2} + (1)\frac{1}{2}= 0$, whereas
the second moment is $\sum k^2 \mu_k = 1$. Next consider
checkerboard with both black and white cells having degrees $0$.
This yields $\mu_{0}=1$ and $\mu_{k}=0, \; k \neq 0$ in the limit
$\ve\to 0$, and therefore $D=0$ as in the previous case. However
now the second moment is $\sum k^2 \mu_k = \Phi = 0$ and it is
optimal from the energy standpoint.}\end{footnote} Thus we need to
choose $d^\ve_j$ that define $\mu^\ve_k$ via (\ref{repres}) so that
the limiting values $\mu_k$ are optimal in the sense of
(\ref{toy_min}).

According to Lemma \ref{Phi} if we represent $D(x)$ for fixed
$x\in\Omega$ as a convex hull of the nearest integers $d$ and
$d+1$, $D(x)=\alpha d +(1-\alpha) (d+1)$, then we must have only
holes with degree $d$ and $d+1$ in a small neighborhood of $x$ and
in this neighborhood $\#\{\text{holes with degree}\
d\}/\#\{\text{holes with degree}\ d+1\}$ should be approximately
equal to $\alpha/(1-\alpha)$.

Recall that the centers of holes $a^\ve_j$ form an $\ve$-lattice and
therefore partition the domain $\Omega$ into fine squares
$\Pi^\ve_j$ of side length $\ve$. In order to construct the
distribution of vortex degrees that yields the optimal probabilities $\mu_k(x)$ in
the  homogenization limit (averaging over the $\ve$ scale), we need a  partition of $\Omega$ on a coarser scale.
To this end
we  fix a  sufficiently
large integer $M$ and consider squares $K_k$ with side length
$(2M+1)\ve$ and centers at points $a^\ve_{k}$, which form an
$(2M+1)\ve$-periodic lattice. Thus, we introduce a mesoscale  $(2M+1)\ve$.
This would allow to the construction of  $\tilde D(x)$ that
approximates $D(x)$ in $L^2(\Omega)$ for small $\ve$ and large $M$.

Consider a square $K_k$ that lies  strictly
 inside $\Omega$ (it contains exactly $(2M+1)^2$ holes).  Let
$D_k$ denote the mean value of $D(x)$ on $K_k$,
$$
D_k:=\frac{1}{(2M+1)^2 \ve^2}\int_{K_k} D(x) dx,
$$
and let $d_k$ be the integer such that $d_k\leq D_k<d_k+1$.  We
next use Lemma \ref{Phi} to recover the degrees of vortices. To
this end we represent   $D_k$ as the convex combination of $d_k$
and $d_k+1$, $D_k=\alpha_k d_k+(1-\alpha_k) (d_k+1)$ ($0\leq
\alpha_k<1$). Thus we need to find a distribution of degrees $d_k$
and $d_{k+1}$ over cells that lie in $K_k$. Since $K_k$ has
finitely many cells, this cannot be done exactly and  we choose
the
largest integer $R>0$ 
such that $R/(2M+1)^2\leq \alpha_k$ and set $d^\ve_j:=d_k$ for $R$
holes $\omega_j^\ve$ in $K_k$ and $d^\ve_j:=d_k+1$ for the remaining
holes in $K_k$. We repeat this procedure for all squares $K_k$
lying strictly inside $\Omega$ and set degrees of the remaining holes
to be zero.  Thus the constructed distribution of degrees defines
the functions $\mu_k^\ve$  by (\ref{repres}), $$ \ve^2\sum
(d_j^\ve)^2=\sum_{l\in Z} l^2 \int_\Omega \mu_l^\ve(x)\,dx.
$$
We claim that for sufficiently small $\ve>0$
\begin{equation}
\label{upperbound}
\tilde E_\ve(\tilde h^\ve))\leq \bar
E_0(D)+\delta_M,
\end{equation}
where $\delta_M\to 0$ as $M\to \infty$. Indeed, due to
boundedness, we have that, up to extracting a subsequence,
$\mu_l^\ve\to\mu_l(x)$ weakly in $L^2(\Omega)$ for all $l\in Z$.
The corresponding limiting vorticity $\tilde D(x)$ (which depends
on $M$) is given by $\tilde D(x)=\sum l\mu_l(x)$. Due to the above
construction of tuples $\{d^\ve_j\}$ on squares, we have two
cases.

First,
when $D(x)$ is not close to integers, ${\rm dist}(D(x),\mathbb{Z})\geq 1/(2M)^2$,
we have $|\mu_d(x)-\alpha|\leq 1/(2M)^2$ and
$|\mu_{d+1}(x)-(1-\alpha)|\leq 1/(2M)^2$, where $\alpha d+(1-\alpha) (d+1)=D(x)$ is the representation of
$D(x)$ as the convex combination of nearest integers. Second, when $|D(x)-d|\leq
1/(2M)^2$ for some integer $d$, we have $|\mu_d(x)-1|\leq 2/(2M)^2$.
 It follows that
\begin{equation}
\label{upperbound_m}
\lim_{\ve\to 0} \ve^2\sum (d_j^\ve)^2=\sum l^2\int_\Omega
\mu_l(x)\leq\int_\Omega \Phi(\tilde D(x))\, dx+C/M^2.
\end{equation}
Analogously by the construction of tuples $\{d^\ve_j\}$ 
we  have $|\tilde D(x)-D(x)|\leq C/M^2$ in $\Omega$. Thus, by
virtue of Lemma \ref{hom_lemma} we obtain (\ref{upperbound}) by
taking the limit $\ve \to 0$ for a given $M$. Now we consider the limit $M
\to \infty$ so that the last term in (\ref{upperbound_m})
vanishes, which proves the desired upper bound.

\subsection{$\Gamma$-convergence theorem}

We summarize the results of this Section in the following

\begin{theorem}
The functionals $\tilde E_\ve$ $\Gamma$-converge to $\bar E_0(D)$ as
$\ve \to 0$.
The limit functional $E_0(D)$ is given by
\begin{equation}
\label{limfunctE0}
E_0(D)=\frac{1}{2}\int_\Omega \left(|\nabla \bar h|^2+(\bar h-\sigma)^2\right)dx+\pi\gamma\int_\Omega\Phi (D(x))dx,
\end{equation}
where $\bar h=\bar h(D)$ is the unique solution of (\ref{hom_eq_bar_h}) and $\Phi(D)=(2k+1)|D|-k-k^2$ if $k\leq|D|< k+1$, $k=0,1,2,\dots$, see Fig \ref{fig}.
\end{theorem}

\begin{figure}
\center{
\includegraphics[width = 4 in] {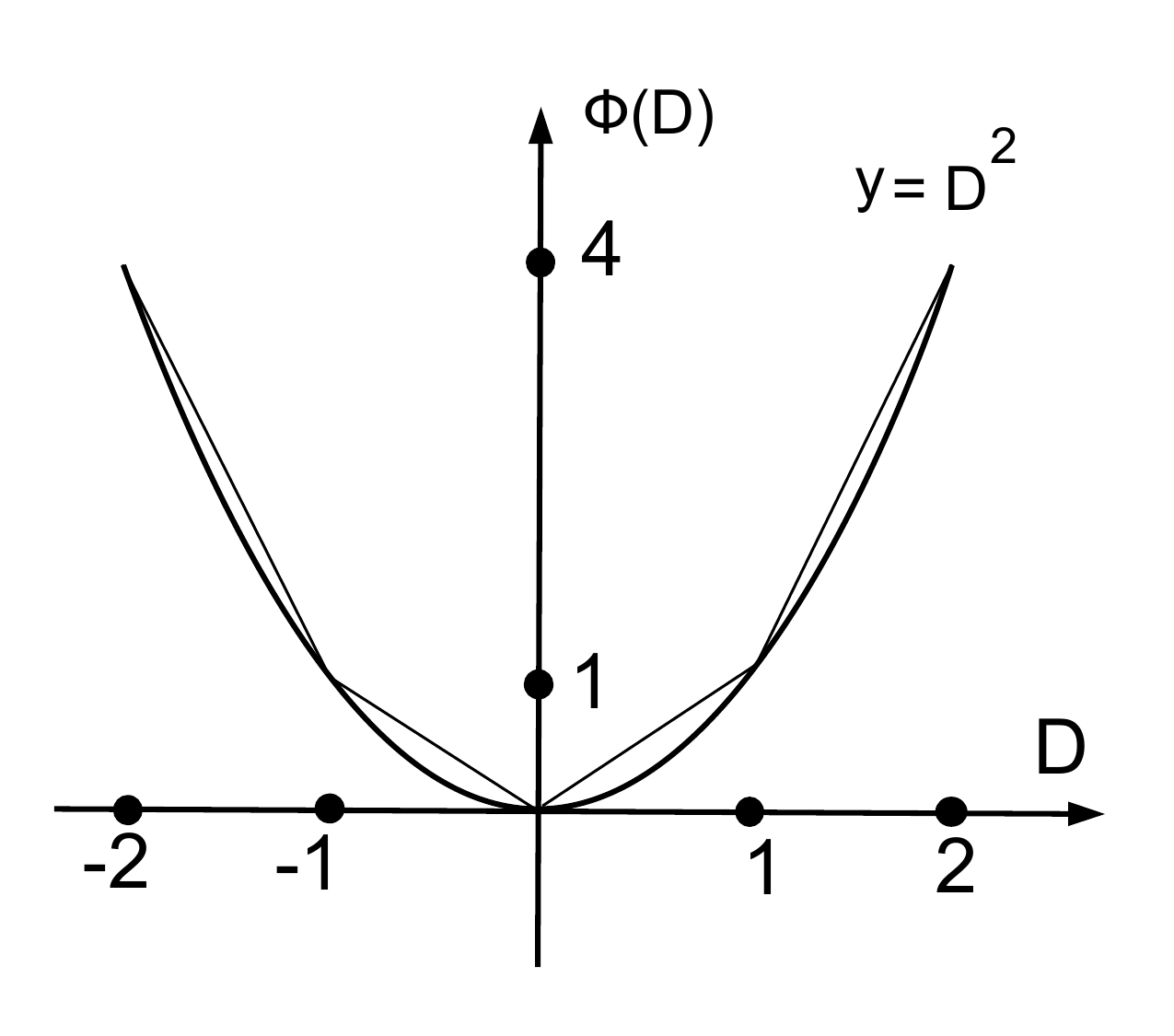}
\caption{\label{fig} The function $\Phi(D)$}}
\end{figure}

This yields the main homogenization result of this work

\begin{theorem}[homogenized  vorticity]
Let  $\{d_j^\ve\}$ be a tuple of integer degrees solving the minimization problem
(\ref{min_prob}), (\ref{harm_map_func})
(i.e. the solution $h^\ve$ of problem (\ref{eq_h}) minimizes the functional
(\ref{reduced_harm_map})),  then
\begin{equation}
\label{main_result}
\ve^2 \sum d_j^\ve\delta_{a_j^\ve}(x)\rightharpoonup D(x),\ \text{in the sense of distributions},
\end{equation}
where $D$ is the unique minimizer of the functional $E_0(D)\
\text{in}\ L^2(\Omega)$, $E_0(D)$ being given by
(\ref{limfunctE0}).
\end{theorem}

\begin{proof} Note that the $\Gamma$-limit functional $E_0(D)$ is strictly convex, continuous and coercive,
therefore it has a unique minimizer. On the other hand, by Lemma
\ref{bound_deg} the weak limit $ \ve^2 \sum
d_j^\ve\delta_{a_j^\ve}(x)\rightharpoonup D(x)$ exists (up to
extracting a subsequence) and $D\in L^2(\Omega)$. Therefore, due
to the classical properties of $\Gamma$-convergence, $D$ is the
unique minimizer of $E_0(D)$. \end{proof}

\begin{remark}
 Note that  this  theorem provides  the homogenized vorticity $D(x)$.
 However, our goal is to describe the distribution of  multiplicities of vortices in terms of  probabilities $\mu_{k} (x)$ of having a vortex of degree $k$  in a neighborhood of a point $x$ .  It can be seen from \eqref{D_represent} that $\mu_{k} (x)$ are not  uniquely defined by $D(x)$. However, the optimality condition \eqref{toy_min} selects  $\mu_{k} (x)$ uniquely so that \eqref{lower_bound_extra} becomes an equality (otherwise \eqref{lower_bound_extra} will be a strict inequality).
The recepie of computing $\mu_{k} (x)$ follows Lemma \ref{Phi}. If $D(x)=d$ is an integer , then the unique minimizing sequence   is  $\mu_d(x)=1$, $\mu_k(x)=0$ \ $\forall k\not=d$. In the case $D(x)$ is
non integer, represent $D(x)$ as the convex hull of two nearest
integers $d$ and $d+1$, $D(x)=\alpha d+(1-\alpha) (d+1)$, then the
unique minimizing sequence is  $\mu_{d}(x)=\alpha$,
$\mu_{d+1}(x)=1-\alpha$, $\mu_k(x)=0$ \ $\forall k\not\in\{d,d+1\}$.
\end{remark}

\section{Analysis of the limit problem via convex duality. Hierarchy of multiplicities }\label{4}

We use convex duality (see, e.g., \cite{ET}) to pass from the
problem
\begin{multline}
\label{gama_lim_direct}
M_\sigma=\min\Bigl\{\frac{1}{2}\int_\Omega\left( |\nabla (h-\sigma)|^2+(h-\sigma)^2
+2\pi\gamma \Phi((-\Delta h+h)/(2\pi))\right)dx;
\Bigr.\\
\Bigr.
(h-\sigma)\in
H^1_0(\Omega),\, -\Delta h+h\in L^2(\Omega)\Bigr\}
\end{multline}
to the dual one
\begin{equation}
\label{gama_lim_inverse}
-M_\sigma=\min\left\{\frac{1}{2}\int_\Omega ( |\nabla f|^2+f^2)\,dx+
\mathcal{F}^{*}(-f);\ f\in H^1_0(\Omega)\right\},
\end{equation}
where $\mathcal{F}^{*}(f)$ is the Legendre transform of the functional
$$
\mathcal{F}(\kappa)=\pi\gamma\int_\Omega\Phi((-\Delta \kappa+\kappa+\sigma)/(2\pi)),
$$
i.e.
$$
\mathcal{F}^{*}(f)=\sup \left\{\int_\Omega ( \nabla f\cdot \nabla\kappa+f\kappa)\,dx -\mathcal{F}(\kappa);
\, \kappa\in H^1_0(\Omega)\right\}.
$$
Due to the fact that $\mathcal{F}(\kappa)$ is lower
semicontinuous, the minimizer $(\bar h-\sigma)$ of
(\ref{gama_lim_direct}) and minimizer $\bar f$ of
(\ref{gama_lim_inverse}) coincide (moreover $M_\sigma$ in
(\ref{gama_lim_direct}) and (\ref{gama_lim_inverse}) is the same).
Thus, we have
\begin{equation}
\label{vort_formula}
2\pi D(x)=-\Delta \bar h+\bar h= -\Delta \bar f+\bar f+\sigma
\end{equation}

The calculation of the Legendre transform $\mathcal{F}^{*}(f)$ is
reduced to the calculation of the Legendre transform $\Phi^*(f)$ of
the function $\pi\gamma\Phi(z/(2\pi))$. Indeed, if we use
integration by parts we derive
\begin{multline*}
  \int_\Omega ( \nabla f\cdot \nabla\kappa+f\kappa)\,dx -\mathcal{F}(\kappa)=
\\
\int_\Omega (-\Delta \kappa +\kappa+\sigma)f\,dx -
\int_\Omega\bigl( \pi\gamma\Phi\bigl((-\Delta \kappa+\kappa+\sigma)/(2\pi)\bigr)-\sigma f\bigr)dx,
\end{multline*}
and therefore
\begin{equation}
\mathcal{F}^{*}(-f)=\int_\Omega( \Phi^*(-f)+\sigma f)dx.
\end{equation}
The Legendre transform  $\Phi^*(f)$ of $\pi\gamma\Phi(z/(2\pi))$ is given by $\Phi^*(f)=0$ for
$|f|\leq\gamma/2$ and $\Phi^*(f)=2\pi k |f|-\pi\gamma k^2$ for
$k\gamma-\gamma/2\leq |f|\leq k\gamma+\gamma/2$.

Thus (\ref{gama_lim_direct}) is equivalent to the problem
\begin{equation}
\label{gama_lim_dual}
\min\Bigl\{\frac{1}{2}\int_\Omega\left( |\nabla f|^2+f^2
+2\Phi^*(f)+2\sigma f\right)dx;
f\in  H^1_0(\Omega)\Bigr\},
\end{equation}
and the limit vorticity is defined in terms of the minimizer $\bar f$ by the formula (\ref{vort_formula}).

The following Lemma shows the monotone dependence of minimizers on
$\sigma$. It is an important tool in the analysis of the
dependence of the vorticity domains on the magnitude of the magnetic
field.

\begin{lemma}
\label{monot_lem} Let $\bar f_\sigma$ be the minimizer of
(\ref{gama_lim_dual}), then $\bar f_\sigma\leq 0$ in $\Omega$ for
every $\sigma>0$, and $\bar f_\alpha\leq \bar f_\beta$ in $\Omega$
if $\alpha>\beta$.
\end{lemma}

\begin{proof}
Set
\begin{equation*}
    \tilde{\Phi}^{*}_{\delta}(f) = \frac{1}{2\delta}\int^{f + \delta}_{f - \delta}
    \Phi^{*}(z)\,dz,
\end{equation*}
since it is an integral of a continuous function,
$\tilde{\Phi}^{*}_{\delta}\in C^1(\mathbb{R})$. Moreover,
$\tilde{\Phi}^{*}_{\delta}$ is convex thanks to the convexity of
$\Phi^{*}(f)$.

%
Let $\tilde{f}_{\sigma}$ be the minimizer of the functional
(\ref{gama_lim_dual}) with $\tilde{\Phi}^{*}_{\delta}$ in place of
${\Phi}^{*}$. It is obvious that this minimizer is a continuous
function. Assume that
the function $\tilde f = \tilde f_\beta - \tilde f_\alpha$ has a
negative minimum. Subtracting the Euler-Lagrange equation for $\tilde
f_\alpha$ from that for $\tilde f_\beta$ we obtain
\begin{equation*}
    -\Delta \tilde f + \tilde f + (\beta - \alpha) +
    (\tilde{\Phi}^{*}_{\delta})^{\prime}(f_\beta) - (\tilde{\Phi}^{*}_{\delta})^{\prime}(f_\alpha) = 0
\end{equation*}
At the minimum point of $\tilde f$ we have that $- \Delta \tilde f
\leq 0$, $\tilde f < 0$, $\beta - \alpha < 0$ and
$(\tilde{\Phi}^{*}_{\delta})^{\prime}(\tilde f_\beta) -
(\tilde{\Phi}^{*}_{\delta})^{\prime}(\tilde f_\alpha) \leq 0$.
Thus we have a contradiction and therefore $\tilde f_\beta \geq
\tilde f_\alpha$. In particular, setting $\beta=0$ we get $\tilde
f_\sigma\leq 0$ for $\sigma>0$.

The result follows by passing to the limit $\delta\to 0$.
\end{proof}

\subsection{Weak magnetic fields: Zero vorticity}

Let us consider weak magnetic fields $h_{ext}^\ve=\sigma/\ve^2$,
such that $\sigma>0$ is small.
It is natural to expect that for
such magnetic fields the minimizer $\bar f_\sigma$ of the problem
(\ref{gama_lim_dual}) satisfies  $-\gamma/2<f_\sigma\leq 0$
therefore $\Phi^*(f_\sigma+v)=0$ in $\Omega$  for every
sufficiently small smooth test function $v\in C_0^\infty(\Omega)$,
and this implies that $\bar f_\sigma$ must solve the problem
\begin{equation}
\begin{cases}
\Delta f=f+\sigma\ \text{in}\ \Omega\\
f=0\ \text{on}\ \partial\Omega.
\end{cases}
\label{no_vort_eq}
\end{equation}
More precisely, the case of zero vorticity is described by

\begin{proposition}
\label{first_sim_prop}
Let $f_1$ be the solution of  (\ref{no_vort_eq}) for $\sigma=1$ and $\gamma$ is defined in
(\ref{scales}). If
\begin{equation}
\sigma\leq\sigma_{{\rm cr}1}:=\frac{\gamma}{2\max |f_1|}
\label{lambda_cr1}
\end{equation}
then the minimizer $\bar f_\sigma$ of (\ref{gama_lim_dual}) is given by $\bar f_\sigma=\sigma f_1$, and,
according to  (\ref{vort_formula}), $D(x)=0$.
\end{proposition}

\begin{proof}
Since
\begin{equation}
\frac{1}{2}\int_\Omega\left( |\nabla f|^2+f^2
+2\Phi^*(f)+2\sigma f\right)dx
\geq \frac{1}{2}\int_\Omega\left( |\nabla f|^2+f^2+2\sigma f\right)dx
\label{compar_funct}
\end{equation}
for every $f$, and (\ref{compar_funct}) becomes an equality if $f$ is the minimizer of the right hand side,
$f=\sigma f_1$, the result follows.
\end{proof}

\subsection{Moderate magnetic fields: Simple vortices}

 If the parameter $\sigma$ begin to  exceed $\sigma_{{\rm cr}1}$,
 then $\sigma f_1$  is no longer the minimizer of  (\ref{gama_lim_dual}),
 and  \eqref{gama_lim_dual} is
 reduced to the  following variational problem
\begin{equation}
\min\left\{\frac{1}{2}\int_\Omega\left( |\nabla f|^2+f^2
+4\pi(|f|-\gamma/2)_{+}+2\sigma f\right)dx;\ f\in H^1_0(\Omega)\right\}.
\label{first_votrtex_f}
\end{equation}

\begin{proposition}
\label{p_mod1}
Let $g_\sigma$ be the minimizer of  (\ref{first_votrtex_f}) and let
\begin{equation}
\sigma_{{\rm cr}2}:=\max\{\sigma>0; \max |g_\sigma|\leq 3\gamma/2\}.
\label{lambda_cr2}
\end{equation}
Consider $\sigma \in (\sigma_{{\rm cr}1}, \sigma_{{\rm cr}2})$,
then the minimizer $\bar f_\sigma$ of (\ref{gama_lim_dual}) coincides with that of  (\ref{first_votrtex_f}).
\end{proposition}
\begin{proof} The proof is identical to that of Proposition \ref{first_sim_prop}. Note that we make use
here of Lemma \ref{monot_lem} to get that the minimizer $g_\sigma$ of  (\ref{first_votrtex_f}) satisfies the
inequality $g_\sigma\geq -3\gamma/2$ in $\Omega$ when $\sigma\leq\sigma_{{\rm cr}2}$.
\end{proof}

In order to describe the limiting vorticity function $D(x)$ we
have to study problem (\ref{first_votrtex_f}) in more details.

\begin{proposition}
\label{p_mod2}
If 
$\sigma_{{\rm cr}1}<\sigma\leq 2\pi + \gamma/2$, then (\ref{first_votrtex_f}) is reduced to the
obstacle problem
\begin{equation}
\min\left\{\frac{1}{2}\int_\Omega\left( |\nabla f|^2+f^2
+2\sigma f\right)dx;\ f\in H^1_0(\Omega),\, |f|\leq \gamma/2 \right\}.
\label{obstacle_problem}
\end{equation}
The minimizer $\bar f_\sigma$ takes the value $-\gamma/2$ on a set
with nonzero Lebesgue measure on $\mathbb{R}^2$. Moreover, the
vorticity $D(x)$ is zero in the domain where $\bar f(x)>-\gamma/2$
and $D(x)=(\sigma-\gamma/2)/(2\pi)$ otherwise.

If 
$\max\{\sigma_{{\rm cr}1},2 \pi + \gamma/2\}<\sigma\leq
\sigma_{{\rm cr}2}$ then $D(x)=1$ when $\bar
f_\sigma(x)<-\gamma/2$ and $D(x)=0$ if $\bar
f_\sigma(x)\geq-\gamma/2$, where $\bar f_\sigma$ is the
minimizer of (\ref{first_votrtex_f}).
\end{proposition}

\begin{remark} It follows from Proposition \ref{p_mod2} that the vorticity $D(x)$ is zero in a subdomain of
$\Omega$ where $\bar f_\sigma>-\gamma/2$ and $0<D(x)\leq 1$  when  $\bar f_\sigma\leq-\gamma/2$.
Moreover, there are two scenarios, if $\sigma_{{\rm cr}1}<\sigma\leq 2\pi + \gamma/2$ then
$0<D(x)<1$ in the set where  $\bar f_\sigma(x)\leq-\gamma/2$ (holes with degrees one and zero),
and if $\max\{\sigma_{{\rm cr}1}, 2\pi + \gamma/2\}<\sigma\leq
\sigma_{{\rm cr}2}$ then $D(x)=1$ when $\bar f_\sigma(x)\leq-\gamma/2$ (all holes in this set
have degree one).
\end{remark}

\begin{proof} Let $\sigma_{{\rm cr}1}<\sigma\leq 2\pi + \gamma/2$ then
$f^2+4\pi(|f|-\gamma/2)_{+}+2\sigma f> \gamma^2/4- \sigma \gamma$ when
$f<-\gamma/2$. It follows that the minimizer $\bar f_\sigma$ of  (\ref{first_votrtex_f})
satisfies the pointwise inequality $\bar f_\sigma\geq -\gamma/2$. Clearly we also have
$\bar f_\sigma\leq 0$. Thus  $\bar f_\sigma$ minimizes (\ref{obstacle_problem}).
If we assume that $\bar f_\sigma$  satisfies the strict
inequality $\bar f_\sigma(x)>-\gamma/2$ for a.e. $x\in \Omega$, then
$-\Delta \bar f_\sigma(x)+\bar f_\sigma(x)+\sigma=0$ for a.e. $x\in \Omega$ and therefore
$\bar f_\sigma$ is the solution of problem (\ref{no_vort_eq}). Since  $\sigma>\sigma_{{\rm cr}1}$
we have a contradiction with the pointwise bound  $\bar f_\sigma\geq -\gamma/2$.

In the case when $\max\{\sigma_{{\rm cr}1}, 2\pi +
\gamma/2\}<\sigma\leq \sigma_{{\rm cr}2}$ we clearly have
$-\Delta \bar f_\sigma(x)+\bar f_\sigma(x)+\sigma=0$ when
$-\gamma/2<\bar f_\sigma\leq 0$, and $-\Delta \bar
f_\sigma(x)+\bar f_\sigma(x)+\sigma=2\pi$  when $\bar
f_\sigma<-\gamma/2$. Thus we only need to show that the level set
$\bar f_\sigma=-\gamma/2$ has zero measure. To this end consider
the set $W=\{ x\in \Omega; \, -\gamma/2\geq\bar
f_\sigma>-\gamma/2-\delta\}$, where $\delta>0$. For sufficiently
small $\delta$ the boundary of $W$ can be divided into two
nonempty parts $S_1=\{x\in\partial W;\, \bar f_\sigma=
-\gamma/2\}$ and $S_2=\{x\in\partial W;\, \bar
f_\sigma=-\gamma/2-\delta\}$. Both sets $S_1$ and $S_2$ have zero
measure. Consider the function $U$ such that $\Delta U=0$ in the
interior of $W$, $U=-\gamma/2$ on $S_1$, and $U=-\gamma/2-\delta$
on $S_2$. By the maximum principle $U<-\gamma/2$ in the interior
of $W$. On the other hand $\bar f_\sigma\leq U$ (otherwise
$\min\{\bar f_\sigma, U\}$ is a minimizer). Thus the level set
$\bar f_\sigma=-\gamma/2$ coincides with $S_1$ and has zero
measure.
\end{proof}

\subsection{Stronger magnetic fields: Multiple vortices}

For $\sigma>\sigma_{{\rm cr}2}$ vortices with multiplicity two
appear. Similarly to the case of simple vortices there are two
scenarios depending on whether $\sigma_{{\rm cr}2}<\gamma/2+2\pi$
or $\sigma_{{\rm cr}2}\geq\gamma/2+2\pi$. Define
\begin{equation}
\sigma_{{\rm cr}3}:=\max\{\sigma>0; \max |g_\sigma|\leq
5\gamma/2\}, \label{third_cr_field}
\end{equation}
where $g_\sigma$ is the minimizer of the problem
\begin{multline}
\min\left\{\frac{1}{2}\int_\Omega\left( |\nabla f|^2+f^2
+4\pi((|f|-\gamma/2)_{+}+(|f|-3\gamma/2)_{+})+2\sigma
f\right)dx;\right.\\
\left. f\in H^1_0(\Omega)\right\}. \label{second_votrtex_f}
\end{multline}

\begin{proposition}
\label{p_str}
If 
$\sigma_{{\rm cr}2}<\sigma\leq 4\pi + 3\gamma/2$ then
(\ref{second_votrtex_f}) is reduced to the obstacle problem
\begin{equation}
\min\left\{\frac{1}{2}\int_\Omega\left( |\nabla f|^2+f^2
+4\pi(|f|-\gamma/2)_{+}+2\sigma f\right)dx;\ f\in
H^1_0(\Omega),\, |f|\leq 3\gamma/2 \right\}.
\label{obstacle_problem1}
\end{equation}
The minimizer $\bar f_\sigma$ takes the value $-3\gamma/2$ on a set
with nonzero Lebesgue measure on $\mathbb{R}^2$. Moreover, the
vorticity $D(x)$ is zero in the domain where $\bar f_\sigma
(x)>-\gamma/2$, $D(x)=1$ when $-3\gamma/2<\bar
f_\sigma(x)<-\gamma/2$ and $D(x)=(\sigma-3\gamma/2)/(2\pi)$ when
$\bar f_\sigma(x)=-3\gamma/2$.

If 
$\max\{\sigma_{{\rm cr}2}, 4\pi + 3\gamma/2\}<\sigma\leq
\sigma_{{\rm cr}3}$ then $D(x)=0$ when $\bar
f_\sigma(x)>-\gamma/2$, $D(x)=1$ if $-3\gamma/2<\bar
f_\sigma(x)<-\gamma/2$ and $D(x)=2$ when
$f_\sigma(x)<-3\gamma/2$, where $\bar f_\sigma$ is the minimizer
of (\ref{second_votrtex_f}).
\end{proposition}

\begin{remark} In the case when $\sigma_{{\rm cr}2}<\sigma\leq 4\pi + 3\gamma/2$ we see
that vortices with multiplicities one and two coexist in the set
where $\bar f_\sigma(x)=-3\gamma/2$, while all holes have degree
one in the domain where $-3\gamma/2<\bar f_\sigma(x)<-\gamma/2$
and zero degree in the domain where $f_\sigma(x)>-\gamma/2$. If
$\max\{\sigma_{{\rm cr}2}, 4\pi + 3\gamma/2\}<\sigma\leq
\sigma_{{\rm cr}3}$ there are three subdomains, where
$f_\sigma(x)>-\gamma/2$, $-3\gamma/2<\bar f_\sigma(x)<-\gamma/2$
and $\bar f_\sigma(x)<-\gamma/2$. All  holes in these domains have
vortices of degrees zero, one, and two correspondingly.
\end{remark}

The proof of this result is similar to the previous ones. Further
increase of the magnetic field leads to vortices with higher
multiplicities in nested subdomains.  Thus our results can be
summarized in the following

\begin{theorem} \label{main_resultth}
There exists a strictly increasing sequence of critical values
$\sigma_{{\rm cr}j}= \sigma_{{\rm cr}j}(\gamma, \Omega)$, $j=1,2,\dots$ such that if
$\sigma _{{\rm cr}j} < \sigma < \sigma_{{\rm cr}(j+1)}$
the limiting vorticity takes constant values
in subsets $\Omega_k\setminus\Omega_{k+1}$, where
$\Omega_k=\Omega_k(\sigma)$, $k=0,1,\dots,j$ are strictly nested sets (vorticity sets)
and $\Omega_0=\Omega$.
Namely, the vorticity $D(x)=0$ in $\Omega_0 \setminus \Omega_1$
and $D(x)=k$ in  $\Omega_{k} \setminus \Omega_{k+1}$, $k\leq j-1$.
Finally, when $x \in \Omega_j$, then there are two scenarios:
(i) if $ \sigma< 2\pi j+(j-1/2)\gamma$
then $(j-1) < D(x) <j$ otherwise (ii) $D(x)=j$.
\end{theorem}

Finally, the following Lemma describes the  dependence of the
vorticity sets on the magnitude of the  external magnetic field
$\sigma$.

\begin{lemma}  If $\sigma \to \infty$, then
for every $k \geq 1$, the domains $\Omega_k(\sigma)$ monotonically
expand to $\Omega$.   On the other hand  if  $\sigma \to
\sigma_{cr1}$, then $\Omega_1(\sigma)$ shrinks to the set of
finitely many points of minima of  the function $f_1$ defined in
\eqref{first_sim_prop}. If the domain $\Omega$ is convex,  then
this set consists of exactly one point and the domain
$\Omega_1(\sigma)$ shrinks to this point.
\end{lemma}

\begin{proof}
Let us seek the minimizer $f_\sigma$ of \eqref{gama_lim_dual} in the form $f_\sigma = \sigma w_\sigma$. Then
$w_\sigma$ minimizes the functional
\begin{equation}
\label{gama_lim_dual_aux}
\min\Bigl\{\frac{1}{2}\int_\Omega\left( |\nabla w_\sigma|^2+w_\sigma^2
+2\Psi_\sigma(w_\sigma)+2 w_\sigma\right)dx;
w_\sigma\in  H^1_0(\Omega)\Bigr\},
\end{equation}
where $$\Psi_\sigma(t)=\frac{1}{\sigma^2}\Phi^*(\sigma t).$$
Note that $\Psi_\sigma(t)$ is a piecewise linear interpolation of the function
$\pi t^2 /\gamma-\pi\gamma/(4\sigma^2)$ at points $k\gamma/\sigma+\gamma/(2\sigma)$, $k\in \mathbb{Z}$.
Therefore $\Psi_\sigma(t)$ converges uniformly to the function  $\pi t^2 /\gamma$ as $\sigma \to \infty$.
It follows that $w_\sigma$ converges strongly in $H^1(\Omega)$ to the minimizer $w$ of the problem
\begin{equation}
\label{gama_lim_dual_aux_lim}
\min\Bigl\{\frac{1}{2}\int_\Omega\left( |\nabla w|^2+(1+2\pi/\gamma)w^2+2 w\right)dx;
w\in  H^1_0(\Omega)\Bigr\},
\end{equation}
which solves $-\Delta w+(1+2\pi/\gamma)w=-1$ in $\Omega$.
By the maximum principle $w<0$ in $\Omega$. Since $f_\sigma = \sigma w_\sigma$, we have that
every sublevel set $f_\sigma\leq -t$ ($t>0$) expands to the domain $\Omega$. This proves that every
set $\Omega_k(\sigma)$ expands to $\Omega$.

The behavior of $\Omega_1(\sigma)$ as $\sigma \to \sigma_{cr1}$ is
a consequence of the structure of the solution of
\eqref{no_vort_eq} described in \cite{SS} .
\end{proof}
\section{Appendix}


 The derivation of the lower bound in Section
\ref{sectionLimVort} makes use of the concept of Young measures.
Recall that a Young measure is a parametrized family of
probability measures $m_x$ associated  with a family of functions
$\phi_\ve(x)$ ($\phi_\ve:\Omega\to \mathbb{R}$) such that $F(\phi_\ve(x))$  converges to $\int_{\mathbb{R}} F(\sigma)\, d m_x (\sigma)$ in $L^{\infty}(\Omega)$ weak star for every bounded continuous function $F(\sigma)$, i.e.,
\begin{equation}
\label{young} \lim_{\ve\to 0}\int_\Omega
F(\phi_\ve(x))\psi(x)dx=\int_\Omega\int_{\mathbb{R}} F(\sigma)\, d m_x(\sigma)\psi(x) dx,\qquad \forall\psi\in L^1(\Omega).
\end{equation}
It is known
(see, e. g., \cite{B}, \cite{P}, \cite{V}) that under some a
priori bounds on the sequence of function $\phi_\ve(x)$ there
exists a family $m_x$ such that (\ref{young}) holds.

In this work we consider a sequence of integer-valued functions
$D^\ve(x)$. We construct a partition of unity $\mu_k(x)$
associated this sequence via (\ref{repres})-(\ref{mu_k_converg}),
so that the corresponding Young measure $m_x$ on $\mathbb{R}^1$
is the sum of
$\delta$-functions centered at integer points, $m_x(\sigma)=\sum
\mu_k(x)\delta_k(\sigma)$. For fixed $k$, the value $\mu_k(x)$
represents the probability to find a hole $\omega_j^\ve$ with
degree $d_j^\ve=k$ in a small vicinity of the point $x$ (i.e.
$\mu_k(x)$ represents the ratio of holes with degree $k$ in a
small neighborhood of $x$ to the total number of holes in this
neighborhood).
\section*{Acknowledgments} The work of  L.Berlyand was supported by NSF grant DMS--1106666. The work of  V. Rybalko was supported in part  by NSF grant DMS--1106666.
Part of this work was done while V. Rybalko was visiting the
Mathematics Department of the Pennsylvania State University. He is
grateful for the hospitality received during his visits. The
authors thank the graduate students O. Iaroshenko and O. Misiats for the
careful reading of the manuscript. The authors are grateful to V.
Vinokour for the stimulating discussions on the physical aspects of
the pinning in composite superconductors.


\medskip
Received January 2012; revised May 2012.
\medskip


\begin{thebibliography}{99}

\bibitem{ASS} (MR1826348) [10.1016/S0021-7824(00)01180-6]
\newblock A. Aftalion, E. Sandier and S. Serfaty,
\newblock \emph{Pinning phenomena in the Ginzburg-Landau model of superconductivity},
\newblock J. Math. Pures Appl. (9), \textbf{80} (2001), 339--372.

\bibitem{AB1} (MR2171205) [10.1063/1.2010354]
\newblock S. Alama and L. Bronsard,
\newblock \emph{Pinning effects and their breakdown for a Ginzburg-Landau model with normal inclusions},
\newblock J. Math. Phys., \textbf{46} (2005), 39 pp.

\bibitem{AB2} (MR2180083) [10.1002/cpa.20086]
\newblock S. Alama and L. Bronsard,
\newblock \emph{Vortices and pinning effects for the Ginzburg-Landau model in multiply connected domains},
\newblock Comm. Pure Appl. Math., \textbf{59} (2006), 36--70.

\bibitem{K2} (MR2476667) [10.3934/cpaa.2009.8.977]
\newblock H. Aydi and A. Kachmar,
\newblock \emph{Magnetic vortices for a Ginzburg-Landau type energy with discontinuous constraint. II},
\newblock Commun. Pure Appl. Anal., \textbf{8} (2009), 977--998.

\bibitem{B}
\newblock E. J. Balder,
\newblock ``Lectures on Young Measures,"
\newblock Cah. de Ceremade, 1995.

\bibitem{BerMilPee06} [10.1103/PhysRevLett.96.207001]
\newblock G. R. Berdiyorov, M. V. Milosevi\'{c} and F. M. Peeters,
\newblock \emph{Novel commensurability effects in superconducting films with antidot arrays},
\newblock Phys. Rev. Lett., \textbf{96} (2006).

\bibitem{DM} (MR2860411) [10.3934/nhm.2011.6.715]
\newblock M. Dos Santos and O. Misiats,
\newblock \emph{Ginzburg-Landau model with small pinning domains},
\newblock Netw. Heterog. Media, \textbf{6} (2011), 715--753.

\bibitem{DMM} (MR2847234) [10.1142/S021919971100449X]
\newblock M. Dos Santos, P. Mironescu and O. Misiats,
\newblock \emph{The Ginzburg-Landau functional with a discontinuous and rapidly oscillating pinning term. Part I : The zero degree case},
\newblock Comm. Contemp. Math., \textbf{13} (2011), 885--914.

\bibitem{Dos11}
\newblock M. Dos Santos,
\newblock \emph{The Ginzburg-Landau functional with a discontinuous and rapidly oscillating pinning term. Part II: The non-zero degree case},
\newblock preprint.

\bibitem{ET} (MR0463993)
\newblock I. Ekeland and R. Temam,
\newblock ``Analyse Convexe et Problemes Variationnels,"
\newblock (French) Collection Etudes Mathematiques. Dunod; Gauthier-Villars, Paris-Brussels-Montreal, Que., 1974.

\bibitem{K1} (MR2674626) [10.1051/cocv/2009009]
\newblock A. Kachmar,
\newblock \emph{Magnetic vortices for a Ginzburg-Landau type energy with discontinuous constraint},
\newblock ESAIM Control Optim. Calc. Var., \textbf{16} (2010), 545--580.

\bibitem{LM} (MR1753480) [10.1007/BF02791255]
\newblock L. Lassoued and P. Mironescu,
\newblock \emph{Ginzburg-Landau type energy with discontinuous constraint},
\newblock J. Anal. Math., \textbf{77} (1999), 1--26.

\bibitem{P} (MR1452107) [10.1007/978-3-0348-8886-8]
\newblock P. Pedregal,
\newblock \doititle{``Parametrized Measures and Variational Principles,"}
\newblock Birkhauser, 1997.

\bibitem{SS} (MR1743433) [10.1016/S0294-1449(99)00106-7]
\newblock E. Sandier and S. Serfaty,
\newblock \emph{Global minimizers for the Ginzburg-Landau functional below the first critical magnetic field},
\newblock Annales IHP, Analyse Non Lin\'eaire, \textbf{17} (2000), 119--145.

\bibitem{V} (MR1079763) [10.1007/BFb0084935]
\newblock M. Valadier,
\newblock \emph{Young measures},
\newblock Methods of Nonconvex Analysis, Lecture Notes Math., Springer, (1990), 152--188.



\end{thebibliography}
\end{document}